\documentclass[preprint,12pt,3p]{elsarticle}

\usepackage{amssymb}
\usepackage{amsmath,amsthm}
\usepackage{mathrsfs}
\usepackage{hyperref}
\usepackage{cleveref}
\usepackage[all,cmtip]{xy}
\usepackage{epsf, graphicx}
\usepackage{latexsym,amsfonts,amsbsy,amssymb}
\usepackage{geometry}
\usepackage{titletoc}
\usepackage{color,xcolor}
\usepackage{rotating}
\usepackage{algorithm}
\usepackage{algorithmic}
\usepackage{amscd}

\makeatletter

\@addtoreset{equation}{section} \makeatother
\newtheorem{theorem}{Theorem}[section]
\newtheorem{lemma}[theorem]{Lemma}
\newtheorem{notation}[theorem]{Notation}
\newtheorem{conjecture}[theorem]{Conjecture}
\newtheorem{corollary}[theorem]{Corollary}
\newtheorem{remark}[theorem]{Remark}
\newtheorem{void}[theorem]{}
\newtheorem{definition}[theorem]{Definition}
\newtheorem{proposition}[theorem]{Proposition}

\def\IBr{{\rm IBr}}
\def\Irr{{\rm Irr}}

\def\br{{\rm br}}

\def\Aut{{\rm Aut}}
\def\Out{{\rm Out}}
\def\Gal{{\rm Gal}}

\def\F{\mathbb{F}}

\def\Z{\mathbb{Z}}
\def\Proj{{\rm Proj}}
\def\DD{D^\Delta}

\def\R{\mathcal{R}}
\def\0{\bar{0}}
\def\1{\bar{1}}

\def\bF{\overline{\F}}
\def\FF{\mathcal{F}}
\def\W{\mathcal{W}}
\def\exp{{\rm exp}}
\def\S{\mathcal{S}}
\def\N{\mathcal{N}}
\def\Proj{\mathrm{Proj}}
\def\PIM{\mathrm{PIM}}
\def\Inn{\mathrm{Inn}}
\def\un{\mathbf{1}}
\newcommand{\lexp}[2]{\setbox0=\hbox{$#2$} \setbox1=\vbox to
	\ht0{}\,\box1^{#1}\!#2}

\makeatletter
\def\ps@pprintTitle{%
\let\@oddhead\@empty
\let\@evenhead\@empty
\def\@oddfoot{\reset@font\hfil\thepage\hfil}
\let\@evenfoot\@oddfoot
}
\makeatother

\begin{document}

\begin{frontmatter}

\title{Alperin's weight conjecture, Galois automorphisms, alternating sums, and functorial equivalences}

\author[label1]{Xin Huang}
\ead{xinhuang@mails.ccnu.edu.cn}

\address[label1]{School of Mathematics and Statistics, Central China Normal University, Wuhan 430079, China}

\author[label2]{Deniz Y\i lmaz}
\ead{d.yilmaz@bilkent.edu.tr}

\address[label2]{Department of Mathematics, Bilkent University, Ankara 06800, Turkey}

\begin{abstract}
We show that functorial equivalences can offer new insight into the blockwise Galois Alperin weight conjecture (BGAWC). Inspired by Kn\"orr and Robinson's work, we first formulate the BGAWC in terms of alternating sums indexed by chains of $p$-subgroups, and we also give a functorial reformulation in the Grothendieck group of diagonal $p$-permutation functors. We prove that these formulations are equivalent. We further show that if a functorial equivalence between a block with abelian defect group and its Brauer correspondent descends to the minimal field of the block, then the BGAWC holds for that block. Finally, we prove that Galois conjugate blocks are functorially equivalent over an algebraically closed field of characteristic zero.
\end{abstract}


\end{frontmatter}


\section{Introduction}\label{s1}

Recently, a substantial amount of literature has been devoted to refinements of the classical local-global counting conjectures taking Galois automorphisms into account. Examples include the Navarro/Galois Alperin--McKay conjecture \cite[Conjecture B]{Nav04}, and the blockwise Navarro/Galois Alperin weight conjecture (see e.g. \cite[Conjecture 2]{H24}). In \cite{BY22}, Bouc and Y\i lmaz introduced the notion of functorial equivalences between blocks. Functorial equivalences have already proved useful in the study of several problems in the block theory of finite groups. For example, in \cite[Theorem 10.6]{BY22}, Bouc and Y\i lmaz investigated the Donovan conjecture from the point of view of functorial equivalences, while in \cite{BBY}, Boltje, Bouc and Y\i lmaz applied the same perspective to Alperin's weight conjecture. In this paper, we show that functorial equivalences may also shed light on the blockwise Galois Alperin weight conjecture (BGAWC) in several ways. To be more precise, we first fix some notation:

\begin{notation}
	{\rm Throughout this paper $p$ is a prime, $\bF_p$ denotes a fixed algebraic closure of $\F_p$, and $\Gamma:=\Gal(\bF_p/\F_p)$. Let $G$ be a group. We denote by $\exp(G)$ the smallest positive integer $n$ (possibly $n=+\infty$) such that $g^n=1$ for all $g\in G$.	If $X$ is a left $G$-set, the stabiliser of an element $x\in X$ is denoted by $G_x$. For any subgroup $H$ of $G$, $X^H$ denotes the set of fixed points of $H$ on $X$, that is, $X^H=\{x\in X\mid hx=x,~\forall~h\in H\}$. By $[G\backslash X]$ we denote a set of representatives of the $G$-orbits of $X$. We denote by $\FF_G$ the set of $p$-subgroups of $G$.  For any field $k$ of characteristic $p$ and any $P\in \FF_G$, we denote by $\br_P^{kG}$ (or $\br_P$) the $P$-Brauer homomorphism $kG\to kC_G(P)$ sending $\sum_{g\in G}\alpha_g g$ to $\sum_{g\in C_G(P)}\alpha_g g$, where $\alpha_g\in k$. A {\it block} of $kG$ is a central primitive idempotent $b$ of $kG$. For any positive integer $n$, we denote by $n_p$ the $p$-part of $n$ and by $n_{p'}$ the $p'$-part of $n$. Unless otherwise specified, all modules are left modules.
	}
\end{notation}

Let $G$ be a finite group and $b$ a central idempotent of $\bF_p G$; we use the convention that $0$ is a central idempotent. Let $\IBr(G,b)$ be the set of characters of simple $\bF_pGb$-modules, and $\W(G,b)$ the set of characters of simple projective $\bF_pGb$-modules; we use the convention that $\IBr(G,0)=\varnothing=\W(G,0)$. The stabiliser $\Gamma_b$ of $b$ in $\Gamma$ acts naturally on both sets $\IBr(G,b)$ and $\W(G,b)$; see Notation \ref{notation:characters of groups} below. The BGAWC can be stated as follows:

\begin{conjecture}[{\cite{Nav04}},{\cite{Turull14}}; cf.\ {\cite[Conjecture 2]{H24}}]\label{conj: GAWC original version}
	Let $G$ be a finite group and $b$ a block of $\bF_p G$. Then there exists a bijection 
	$$\IBr(G,b)\to \bigsqcup_{P\in[G\backslash\FF_G]} \W(N_G(P)/P,\bar{b}_P)$$ 
	commuting with the action of $\Gamma_b$, where $\bar{b}_P$ is the image in $\bF_p[N_G(P)/P]$ of $b_P={\rm br}_P^{\bF_p G}(b)$, regarded as a sum of blocks of $\bF_p[N_G(P)/P]$.
\end{conjecture} 

See \cite{,Turull14,H24,DH,DHZ,FFZ} for some recent results on Conjecture \ref{conj: GAWC original version}, and see Proposition \ref{prop:equivalent reformulations of BGAWC} below for further equivalent reformulations of Conjecture \ref{conj: GAWC original version}. In \cite{KR}, Kn\"orr and Robinson reformulated the blockwise Alperin weight conjecture (BAWC) in terms of alternating sums indexed by chains of $p$-subgroups. Their reformulation led to a wide range of more precise conjectures, including Dade's ordinary and projective conjectures \cite[Conjecture 6.12.9]{Lin18b} and Robinson's conjecture \cite[Conjecture 6.12.10]{Lin18b}. Kn\"orr and Robinson's reformulation also suggests the possibility of a structural explanation for the BAWC; see the beginning of \cite[Section 10.7]{Lin18b} for an exposition. For this reason, it may be useful to reformulate Conjecture \ref{conj: GAWC original version} in terms of alternating sums indexed by chains of $p$-subgroups as well. Again, we need some notation:

\begin{notation}\label{notation:chain}
	{\rm Let $G$ be a finite group and $b$ a block of $\bF_pG$. Denote by $\S_G$ the set of strictly ascending chains $\sigma=(1=P_0<P_1<\cdots<P_n)$ of $p$-subgroups of $G$, and by $\N_G$ the subset of $\S_G$ consisting of those chains $\sigma$ such that every $P_i$ is normal in the maximal term $P_n$. Denote by $G_\sigma$ the stabiliser in $G$ of $\sigma$; we have $G_\sigma=\cap_{i=0}^{n} N_G(P_i)$. For any $\sigma=(1=P_0<P_1<\cdots<P_n)\in \S_G$, we denote $|\sigma|:=n$ and denote $b_\sigma:={\rm br}_{P_n}^{\bF_p G}(b)$. Since $\br_{P_n}^{\bF_pG}$ commutes with the action of $N_G(P_n)$, hence the action of $G_\sigma$. It follows that $b_\sigma$ is an idempotent in $Z(\bF_pG_\sigma)$, possibly $0$.
}	
\end{notation}

We now provide a reformulation of Conjecture \ref{conj: GAWC original version} in terms of alternating sums - a refinement of Kn\"orr and Robinson's reformulation of the original BAWC.

\begin{theorem}[Galois refinement of {\cite[Theorem 3.8]{KR}} or {\cite[Theorem 10.7.6]{Lin18b}}]\label{thm Galois refinement}
	The following statements are equivalent:
	\begin{enumerate}[{\rm (i)}]
		\item For any finite group $G$ and any block $b$ of $\bF_pG$, Conjecture \ref{conj: GAWC original version} holds. 
		\item For any finite group $G$, any block of $\bF_pG$ with a nontrivial defect group, and any subgroup $T$ of $\Gamma_b$, we have
		$$\sum_{\tau\in[G\backslash\S_G]}(-1)^{|\tau|}|\IBr(G_\tau,b_\tau)^T|=0,$$
		where $b_\tau$ is defined as in Notation \ref{notation:chain}.
		\item For any finite group $G$, any block of $\bF_pG$ with a nontrivial defect group, and any subgroup $T$ of $\Gamma_b$, we have
		$$\sum_{\tau\in[G\backslash\N_G]}(-1)^{|\tau|}|\IBr(G_\tau,b_\tau)^T|=0.$$
		\item For any finite group $G$, any idempotent $b$ of $Z(\bF_pG)$, and any subgroup $T$ of $\Gamma_b$, we have
		$$\sum_{\tau\in[G\backslash\N_G]}(-1)^{|\tau|}|\IBr(G_\tau,b_\tau)^T|=|\W(G,b)^T|.$$
		In particular, if $b$ is a block of $\bF_pG$ of defect zero, then the right-hand side equals $1$.
	\end{enumerate}
	
\end{theorem}

Theorem \ref{thm Galois refinement} will be proved in Section~\ref{section: On Galois automorphisms and alternating sums}. We also introduce in Corollary~\ref{cor GaloisAlperinFunctorial} (ii) a functorial reformulation of the BGAWC, expressed in terms of diagonal $p$-permutation functors, and we show that this reformulation is equivalent to the alternating sum version in Theorem \ref{thm Galois refinement} (ii). This provides the first indication of how functorial equivalences (see \ref{void:notaions related to functorial equiavlences}) may shed light on Conjecture \ref{conj: GAWC original version}.

\begin{remark}
{\rm Similar to \cite[Remark~10.7.2]{Lin18b}, our proof of Theorem~\ref{thm Galois refinement} does not show that, for a given block, the validity of one of the statements~(i) or~(ii) implies the validity of the other. However, for a particular block, Theorem~\ref{thm Galois refinement} (ii) holds if and only if Corollary~\ref{cor GaloisAlperinFunctorial} (ii) holds for that block.
}
\end{remark}

\begin{definition}[cf. {\cite[Definition 1.8]{Kessar_Linckelmann}}]
	{\rm Let $G$ be a finite group and $b$ an idempotent in $Z(\bF_pG)$. Write $b=\sum_{g\in G}\alpha_gg$, where $\alpha_g\in \bF_p$. We denote by $\F_p[b]$ the smallest subfield of $\bF_p$ containing all the coefficients $\{\alpha_g\mid g\in G\}$ and call it the {\it minimal field} of $b$. By definition, for any subfield $k$ of $\bF_p$ containing $\F_p[b]$, we have $\Gal(k/\F_p)_b=\Gal(k/\F_p[b])$. 
	}
\end{definition}

For a second way in which functorial equivalences may offer insight into Conjecture \ref{conj: GAWC original version}, we prove the following theorem:

\begin{theorem}\label{theorem:funtorial equivalences and Brauer character bijections}
	Let $R$ be a commutative ring with $1$ containing $\Z$ as a unitary subring. Let $G$ and $H$ be finite groups, and let $b\in Z(\bF_pG)$ and $c\in Z(\bF_pH)$ be idempotents. Assume that $\F_p[b]=\F_p[c]$ and denote this common field by $k$. Assume further that the $k$-pairs $(G,b)$ and $(H,c)$ are functorially equivalent over $R$. Then there exists a bijection $I:\IBr(G,b)\to \IBr(H,c)$ such that $\lexp{\sigma}I(\varphi)=I(\lexp{\sigma}\varphi)$ for all $\varphi\in \IBr(G,b)$ and all $\sigma\in \Gamma_b=\Gal(\bF_p/k)$.
\end{theorem}
 
As a corollary, we show that if the Galois-descent refinement of Brou{\'e}'s conjecture holds at the level of functorial equivalences, then Conjecture~\ref{conj: GAWC original version} follows.

\begin{corollary}\label{cor BrouetoGaloisAlperin}
	Let $R$ be a commutative ring with $1$ containing $\Z$ as a unitary subring. Let $G$ be a finite group, and let $b$ a block of $\bF_pG$ with an abelian defect group $P$. Let $c\in \bF_pN_G(P)$ be the Brauer correspondent of $b$ and set $k=\F_p[b]=\F_p[c]$. Assume that the $k$-pairs $(G,b)$ and $(N_G(P),c)$ are functorially equivalent over $R$. Then Conjecture \ref{conj: GAWC original version} holds for $b$.
\end{corollary}

Theorem \ref{theorem:funtorial equivalences and Brauer character bijections} and Corollary \ref{cor BrouetoGaloisAlperin} will be proved in Section \ref{section: Funtorial equivalences and Brauer character bijections} as applications of a more general result (Theorem \ref{theorem:Galois-equivariant bijections for general algebras}) for arbitrary algebras. Corollary \ref{cor BrouetoGaloisAlperin} suggests that Galois descent properties of functorial equivalences may play a role in verifying Conjecture \ref{conj: GAWC original version}. We take a first step in this direction by showing that Galois conjugate blocks are functorially equivalent. This result is expected to be useful in establishing Galois descent phenomena for functorial equivalences.

\begin{theorem}\label{theorem: Galois conjugate blocks are func equ}
Let $k$ be an algebraically closed field of characteristic $p$ and $\F$ an algebraically closed field of characteristic zero. Let $(G,b)$ be a group-block pair over $k$ and let $\sigma\in\Gamma=\Gal(k/\F_p)$. Then the pairs $(G,b)$ and $(G,\sigma(b))$ are functorially equivalent over $\F$. 
\end{theorem}

Theorem \ref{theorem: Galois conjugate blocks are func equ} will be proved in Section \ref{section: Galois conjugate blocks are functorially equivalent} using the multiplicity formula \cite[Theorem 8.22 (b)]{BY22} for simple diagonal $p$-permutation functors in a block functor.

\section{The BGAWC and alternating sums}\label{section: On Galois automorphisms and alternating sums}
\begin{notation}\label{notation:characters of groups}
{\rm (i) Let $k\subseteq k'$ be an extension of fields. Let $G$ be a finite group and $V$ a finite-dimensional $k'G$-module. The {\it character of $V$} is the class function $\chi_V: G\to k'$, sending $g\in G$ to the trace ${\rm tr}(\rho(g))$ of the linear isomorphism $\rho(g)$ of $V$ defined by $\rho(g)(v)=gv$ for all $v\in V$. For $\sigma\in \Gal(k'/k)$, we denote by $\lexp{\sigma}V$ the $k'G$-module which is equal to $V$ as a module over the subring $kG$ of $k'G$ such that $x$ acts on $\lexp{\sigma}V$ as $\sigma^{-1}(x)$ for all $x\in k'$. For a character $\chi$ of $G$ over $k'$, denote by $\lexp{\sigma}\chi:G\to k'$ the character such that $\lexp{\sigma}\chi(g)=\sigma(\chi(g))$ for all $g\in G$. It is well known (and easy to check) that the character of $\lexp{\sigma}V$ is $\lexp{\sigma}\chi$.

(ii) A map $\chi:G\to k$ is called a(n) {\it (irreducible) character of $G$ over $k$} if $\chi$ is the character of some finite-dimensional (simple) $kG$-module. For $b$ an idempotent in $Z(kG)$, we denote by  $\Irr(kGb)$ the set of irreducible characters of $G$ afforded by simple $kGb$-modules. We adopt the convention that for $b=0$, $\Irr(kGb)=\varnothing$, and a sum indexed by the empty set is zero.  We denote by $\W(kGb)$ the set of characters of simple projective $kGb$-modules; we can identify $\W(kGb)$ with the set of defect zero blocks occurring in a primitive decomposition of $b$ in $Z(kG)$. Let $b$ be an idempotent in $Z(\bF_b G)$, we denote $\Irr(\bF_pGb)$ (resp. $\W(\bF_pGb)$) by $\IBr(G,b)$ (resp. $\mathcal{W}(G,b)$).  Now we see that $\Gamma$ acts on the set $\Irr(\bF_pG)$. Let $\Gamma_b:=\{\sigma\in\Gamma\mid \sigma(b)=b\}$. Then $\Gamma_b$ acts on the sets $\IBr(G,b)$ and $\W(G,b)$. }
\end{notation}

\begin{remark}
{\rm Let $t$ be a positive integer which is not divisible by $p$, and let $\xi\in \bF_p$ be a primitive $t$-th root of unity. Denote by $\Gamma_t$ the group $\Gal(\F_p[\omega]/\F_p)$. Clearly $\Gamma_t$ is independent of the choice of $\xi$.  By elementary field theory, if $t'$ is another positive integer which is not divisible $p$ and which is divisible by $t$, then the restriction map $\Gamma_{t'}\to \Gamma_t$ is surjective. Moreover, it is well known (and easy to prove by using the Zorn lemma) that the restriction map $\Gamma\to \Gamma_t$ is surjective.
 
}
\end{remark} 

The following lemma is a useful criterion for permutation isomorphisms:

\begin{lemma}[{\cite[Lemma 13.23]{Isaacs}}]\label{lemma:permutation isomorphic}
Let $S$ be a (possibly infinite) group and let $X$, $Y$ be two finite $S$-sets. Suppose that for every subgroup $T$ of $S$, we have $|X^T|=|Y^T|$. Then $X$ and $Y$ are isomorphic as $S$-sets.
\end{lemma}
 
By Lemma \ref{lemma:permutation isomorphic}, we easily obtain several equivalent reformulations of Conjecture \ref{conj: GAWC original version}: 
 
\begin{proposition}\label{prop:equivalent reformulations of BGAWC}
Let $G$ be a finite group and $b$ a block of $\bF_pG$.  Then the following are equivalent:
\begin{enumerate}[{\rm (i)}]
	\item Conjecture \ref{conj: GAWC original version} holds for the block $b$.
	\item There exists a bijection 
	$$\IBr(G,b)\to \bigsqcup_{P\in[G\backslash\FF_G]} \W(N_G(P)/P,\bar{b}_P)$$ 
	commuting with the action of $(\Gamma_t)_b$, where $t=\exp(G)_{p'}$.
	\item For any subgroup $T$ of $\Gamma_b$, we have 
	$$|\IBr(G,b)^T|=\sum_{P\in[G\backslash\FF_G]}|\W(N_G(P)/P,\bar{b}_P)^T|.$$
	\item  For any subgroup $T$ of $(\Gamma_t)_b$ where $t=\exp(G)_{p'}$, we have 
	$$|\IBr(G,b)^T|=\sum_{P\in[G\backslash\FF_G]}|\W(N_G(P)/P,\bar{b}_P)^T|.$$
	
\end{enumerate}

\end{proposition}

\begin{proof}
Let $\omega$ be a primitive $t$-th root of unity in $\bF_p$ and let $k:=\F_p[\omega]$. Since $k$ is a splitting field for all subgroups of $G$, the values of any irreducible character of $G$ over $\bF_p$ are contained in $k$. Let $\sigma\in \Gamma_b$ and let $\tilde{\sigma}$ be the image of $\sigma$ under the restriction map $\Gamma_b\to (\Gamma_t)_b$. Then we have the following commutative diagrams:
$$\xymatrix{
\IBr(G,b)\ar@{=}[r]\ar[d]^{\sigma} & \IBr(kGb)\ar[d]^{\tilde{\sigma}}  \\
\IBr(G,b)\ar@{=}[r] &  \IBr(kGb)
}~~~~   \xymatrix{
\W(N_G(P)/P,\bar{b}_P)\ar@{=}[r]\ar[d]^{\sigma} & \W(k(N_G(P)/P)\bar{b}_P)\ar[d]^{\tilde{\sigma}}  \\
\W(N_G(P)/P,\bar{b}_P)\ar@{=}[r] &  \W(k(N_G(P)/P)\bar{b}_P)
}$$
where $P\in [G\backslash\FF_G]$. This proves the equivalence of (i) and (ii).
 The equivalence of (i) and (iii) and the equivalence of (ii) and (iv) follow from Lemma \ref{lemma:permutation isomorphic}.	
\end{proof}

In the rest of this section we will use Notation \ref{notation:chain}.

\begin{lemma}[Galois refinement of {\cite[Lemma 10.7.4]{Lin18b}}]\label{lemma: Galois refinement of 10.7.4}
Let $G$ be a finite group, $n$ a positive integer and $\sigma=(1=P_0<P_1<\cdots<P_n)$ a chain in $\N_G$. Set $P=P_1$. Let $b$ be an idempotent in $Z(\bF_p G)$ such that $\br_P^{\bF_pG}(b)\neq 0$ and denote by $c$ the image of $\br_P^{\bF_pG}(b)$ in $\bF_p[N_G(P)/P]$. Set $\bar{\sigma}=(1<P_2/P<\cdots<P_n/P)\in \S_{N_G(P)/P}$. 
\begin{enumerate}[{\rm (i)}]
	\item Write $N_P=N_G(P)/P$. Then $G_\sigma/P=(N_P)_{\bar{\sigma}}$.
	\item  The image of $b_\sigma=\br_{P_n}^{\bF_p G}(b)$ in $\bF_p[N_G(P_n)/P]$ is equal to $c_{\bar{\sigma}}=\br_{P_n/P}^{\bF_p[N_G(P)/P]}(c)$. In particular, for any subgroup $T$ of $\Gamma_b$, we have
	$$|\IBr(G_\sigma,b_\sigma)^T|=|\IBr(G_\sigma/P,c_{\bar{\sigma}})^T|.$$
\end{enumerate}
\end{lemma}

\begin{proof}
Statement (i) is a trivial exercise. Except for the last equality, everything in statement (ii) is proved in the proof of \cite[Lemma 10.7.4]{Lin18b}. By the proof of \cite[Lemma 10.7.4]{Lin18b}, the inflation map $\IBr(G_\sigma/P) \to \IBr(G_\sigma)$ induces a bijection $\IBr(G_\sigma/P,c_{\bar{\sigma}}) \to \IBr(G_\sigma,b_\sigma)$. By definition, one easily sees that this bijection commutes with the action of $\Gamma_b$. Now by Lemma \ref{lemma:permutation isomorphic} we obtain the desired equality $|\IBr(G_\sigma,b_\sigma)^T|=|\IBr(G_\sigma/P,c_{\bar{\sigma}})^T|.$
\end{proof}

\begin{lemma}[Galois refinement of {\cite[3.6, 3.7]{KR}} or {\cite[Lemma 10.7.5]{Lin18b}}]\label{lemma: Galois refinement of 10.7.5}
Let $G$ be a finite group and $b$ a block of $\bF_pG$. For $P$ a nontrivial $p$-subgroup of $G$, set $N_P=N_G(P)/P$ and denote by $c_P$ the image of $\br_P^{\bF_pG}(b)$ in $\bF_pN_P$. For any subgroup $T$ of $\Gamma_b$, we have
$$\sum_{1\neq\sigma\in [G\backslash\N_G]}(-1)^{|\sigma|-1}|\IBr(G_\sigma,b_\sigma)^T|$$
$$=\sum_{1\neq P\in [G\backslash\FF_G]} \left(|\IBr(N_P,c_P)^T|-\sum_{1\neq\tau\in [N_P\backslash\N_{N_P}]}(-1)^{|\tau|-1}|\IBr((N_P)_\tau,(c_P)_\tau)^T|\right).$$ If $b$ has defect zero, then this sum is zero.
\end{lemma}

\begin{proof}
We proceed by an argument analogous to that in the proof \cite[Lemma 10.7.5]{Lin18b}. The chain of length $1$ in the sum on the left side yields the summand $\sum_{1\neq P\in[G\backslash\FF_G]}|\IBr(N_G(P),b_P)^T|$, where $b_P=\br_P(b)$, regarded as a (possibly empty) sum of blocks of $\bF_pN_G(P)$. Since the kernel of the canonical algebra homomorphism $\bF_pN_G(P)\to \bF_pN_P$ is in the Jacobson radical of $\bF_pN_G(P)$, the inflation map $\IBr(N_P,c_P)\to \IBr(N_G(P),b_P)$ is a bijection. By definition one easily sees that this bijection commutes with the action of $\Gamma_b$. Hence by Lemma \ref{lemma:permutation isomorphic}, we have
$$\sum_{1\neq P\in[G\backslash\FF_G]}|\IBr(N_G(P),b_P)^T|=\sum_{1\neq P\in [G\backslash\FF_G]}|\IBr(N_P,c_P)^T|.$$

The chains $\sigma=(1=P_0<P_1<\cdots<P_{|\sigma|})$ with $|\sigma|>1$ in the sum of the left side are partitioned according to their second term $P_1$. Set $\bar{\sigma}=1<P_2/P_1<\cdots<P_{|\sigma|}/P_1$. By \cite[Lemma 10.7.3]{Lin18b}, two chains $\sigma$ and $\sigma'$ in $\N_G$ with the same second term (say $P_1=P'_1=P$) are $G$-conjugate if and only if $\bar{\sigma}$ and $\bar{\sigma}'$ are $N_{P}$-conjugate. In that process (i.e. partition according to the second term and then passage from $\sigma$ to $\bar{\sigma}$), chains are shortened by one term. So we have to adjust the signs, accounting for the minus signs in front of the sums  $\sum_{\tau\in [N_P\backslash\N_{N_P}]}(-1)^{|\tau|-1}|\IBr((N_P)_\tau,(c_P)_\tau)^T|$. Finally, the fact that $|\IBr((N_P)_\tau,(c_P)_\tau)^T|$ shows up on the right side follows from Lemma \ref{lemma: Galois refinement of 10.7.4}. If $b$ has defect zero, then $\br_P(b)=0$ for any nontrivial $p$-subgroup $P$ of $G$, whence the last statement.
\end{proof}

\begin{proof}[Proof of Theorem \ref{thm Galois refinement}]
We proceed by an argument analogous to that in the proof \cite[Theorem 10.7.6]{Lin18b}. The equivalence of (ii) and (iii) follows from \cite[Proposition 3.3]{KR} or \cite[Proposition 8.13.13]{Lin18b}. If $b$ is a block of defect zero, then $\IBr(G,b)=\W(G,b)$, and hence $|\IBr(G,b)^T|=|\W(G,b)^T|$. By Lemma \ref{lemma: Galois refinement of 10.7.5}, the sum $\sum_{1\neq\tau\in [G\backslash \N_G]}(-1)^{|\tau|}|\IBr(G_\tau,b_\tau)^T|$ is 0, and hence 
$$\sum_{\tau\in [G\backslash \N_G]}(-1)^{|\tau|}|\IBr(G_\tau,b_\tau)^T|=|\IBr(G,b)^T|.$$ So in this case, the equality in (iv) holds.

 If $b$ is a block of positive defect, then $\W(G,b)=\varnothing$. Thus if $b$ is an arbitrary idempotent in $Z(\bF_pG)$, then $|\W(G,b)^T|$ counts the number of defect zero blocks in a primitive decomposition of $b$ in $Z(\bF_pG)$ that are fixed by $T$. The equivalence of (iii) and (iv) is an immediate consequence.
 
 Suppose that (iv) holds. Let $G$ be a finite group and $b$ a block of $\bF_pG$ of positive defect. Since (iv) holds, it follows that the left side in Lemma \ref{lemma: Galois refinement of 10.7.5} in $|\IBr(G,b)^T|$. The right side is, again by (iv), equal to $\sum_{1\neq P\in [G\backslash \FF_G]}|\W(N_P,c_P)^T|$, and this is the statement of a reformulation of Conjecture \ref{conj: GAWC original version} for the block $b$, see Proposition \ref{prop:equivalent reformulations of BGAWC} (iii). Thus (iv) implies (i).
 
 Suppose that (i) holds. We show (iv) by induction on $|G|$. By the discussion in the first two paragraphs, we may assume that $b$ is a block of $\bF_pG$ with positive defect. Note that for $P$ a nontrivial $p$-subgroup of $G$, the group $N_P=N_G(P)/P$ in Lemma \ref{lemma: Galois refinement of 10.7.5} has order strictly smaller that $|G|$. Thus by (iv) applied to the sum of blocks $c_P$ of $\bF_pN_P$, the right side in Lemma \ref{lemma: Galois refinement of 10.7.5} is equal to
 $\sum_{1\neq P\in [G\backslash \FF_G]}|\W(N_P,c_P)^T|$. But since Conjecture \ref{conj: GAWC original version} is assumed to hold, by Proposition \ref{prop:equivalent reformulations of BGAWC} (iii), $\sum_{1\neq P\in [G\backslash \FF_G]}|\W(N_P,c_P)^T|=|\IBr(G,b)^T|$. Now the equality in Lemma \ref{lemma: Galois refinement of 10.7.5} becomes
 $$\sum_{1\neq\sigma\in [G\backslash\N_G]}(-1)^{|\sigma|-1}|\IBr(G_\sigma,b_\sigma)^T|=|\IBr(G,b)^T|,$$
 and this shows that (iv) holds also for $b$.
\end{proof}

\section{Galois refinement of functorial Alperin weight conjecture}\label{section: Galois refinement of functorial Alperin weight conjecture}

Let $k$ be a field of characteristic $p>0$, let $\F$ be an algebraically closed field of characteristic zero and let $R$ be a commutative ring with $1$. Let $G$ and $H$ be finite groups. We first recall the notions of diagonal $p$-permutation functors and functorial equivalences. We refer the reader to  \cite{BY22} and \cite{BY20} for more details. Note that in \cite{BY22} and \cite{BY20}, the field $k$ is assumed to be algebraically closed; however, these notions can be defined over arbitrary field $k$. 

\begin{void}\label{void:notaions related to functorial equiavlences}
{\rm (i) By a \textit{group-idempotent pair} over $k$, or simply a \textit{$k$-pair}, we mean a pair $(G,b)$ consisting of a finite group $G$ and a central idempotent $b\in Z(kG)$. When $b$ is a block idempotent, such pairs are called \textit{group-block pairs} in \cite{BBY}. The notion of functorial equivalence was originally defined for group-block pairs, but it extends naturally to group-idempotent pairs.

(ii) A \textit{$\DD$-pair} is a pair $(L,u)$ where $L$ is a $p$-group and $u\in \Aut(L)$ is an automorphism of order prime to $p$. We write $L\langle u\rangle$ for the semidirect product $L\rtimes \langle u\rangle$. If $(M,v)$ is another $\DD$-pair, then an isomorphism between the pairs $(L,u)$ and $(M,v)$ is a group isomorphism $L\langle u\rangle\to M\langle v\rangle$ that sends $u$ to an element conjugate to $v$. We denote by $\Aut(L,u)$ the automorphism group of the pair $(L,u)$ and by $\Out(L,u)$ the group of outer automorphisms $\Aut(L,u)/\Inn(L\langle u\rangle)$ of $(L,u)$; see \cite[Notation~6.8 and Lemma~6.10]{BY22}. 
	
(iii) Let $b\in Z(kG)$ and $c\in Z(kH)$ be central idempotents. A $p$-permutation $(kGb,kHb)$-bimodule is called \textit{diagonal} if its indecomposable direct summands have twisted diagonal vertices as subgroups of $G\times H$. We denote by $T^\Delta(kGb,kHc)$ the Grothendieck group of diagonal $p$-permutation $(kGb,kHc)$-bimodules with respect to split short exact sequences, and we set $$RT^\Delta(kGb,kHc):=R\otimes_\Z T^\Delta(kGb,kHc).$$

We denote by $Rpp_k^\Delta$ the category whose objects are finite groups and whose morphisms from $G$ to $H$ are given by $RT^\Delta(kH,kG)$. The composition is induced from the tensor product of bimodules. An $R$-linear functor from $Rpp_k^\Delta$ to the category of $R$-modules is called a \textit{diagonal $p$-permutation functor} over $R$. The category of diagonal $p$-permutation functors over $R$ is denoted by $\FF^\Delta_{Rpp_k}$; it is an abelian category.

(iv) We denote $$RT^\Delta_{(G,b)}:=RT^\Delta(-,kGb),$$ and we say that the $k$-pairs $(G,b)$ and $(H,c)$ are \textit{functorially equivalent over $R$}, if the diagonal $p$-permutation functors $RT^\Delta_{(G,b)}$ and $RT^\Delta_{(H,c)}$ are isomorphic. By Yoneda's Lemma, this is equivalent to the existence of elements $\eta\in RT^\Delta(kGb,kHc)$ and $\omega\in RT^\Delta(kHc,kGb)$ such that 
\begin{align*}
	\eta \cdot_H\omega =[kGb]\,\, \text{in } RT^\Delta(kGb,kGb)\quad\text{and}\quad \omega \cdot_G\eta=[kHc]\,\, \text{in } RT^\Delta(kHc,kHc)\,,
\end{align*}
where $\cdot_G$ and $\cdot_H$ denote the composition in the category $Rpp_k^\Delta$; see \cite[Lemma~10.2]{BY22}.

(v) If $k$ is algebraically closed, then, by \cite[Corollary~6.15]{BY22}, the category $\FF^\Delta_{\F pp_k}$ of diagonal $p$-permutation functors over $\F$ is semisimple. Moreover, the simple functors $S_{L,u,V}$ are parametrized by the isomorphism classes of triples $(L,u,V)$, where $(L,u)$ is a $\DD$-pair and $V$ is a simple $\F\Out(L,u)$-module. 

(vi) We denote by $\Proj(kHc,kGb)$ the Grothendieck group of projective $(kHc,kGb)$-bimodules, and by $R\Proj(-,kG)$ the diagonal $p$-permutation functor sending a finite group $H$ to $$R\Proj(kH,kG)=R\otimes_\Z\Proj(kH,kG).$$ Note that when $R=\F$ and $k$ is algebraically closed, by \cite[Lemma~3.5]{BY24}, one has $$\F\Proj(-,kG)=|\IBr(G)|\cdot S_{\un,1,\F}$$ in the Grothendieck group $K_0(\FF^\Delta_{\F pp_k})$ of diagonal $p$-permutation functors. 
}
\end{void}

\begin{notation}
{\rm Let $b\in Z(\overline{\F}_p(G))$ be a central idempotent and let $T\le \Gamma_b$. Let $\PIM(G,b)^T$ denote the set of $T$-fixed projective indecomposable right $\bF_pGb$-modules and let $\F\Proj(kGb)^T$ denote the $\F$-span of $\PIM(G,b)^T$. We denote by $\F\Proj(-,kGb)^T$ the subfunctor of $\F T^\Delta_{(G,b)}$ generated by the set $\PIM(G,b)^T$ viewed as a subset of 
	$$\F T^\Delta_{(G,b)}(\un)=\F T^\Delta(\un,kGb)=\F\Proj(\un,kGb)\cong \F \Proj(kGb).$$ 
	In other words, for any finite group $H$, we have
\begin{align*}
\F\Proj(H,kGb)^T=\sum_{P\in \PIM(G,b)^T} \{X\otimes_k P\in \F T^\Delta(kH,kGb)\mid X\in \F T^\Delta(kH,\un)\}\,.
\end{align*}
}
\end{notation}

\begin{lemma}\label{lem Galoisstablefunctor}
Let $b$ be a central idempotent of $\bF_pG$ and let $T\le \Gamma_b$ be a subgroup. One has
\begin{align*}
\F\Proj(-,kGb)^T=|\IBr(G,b)^T|\cdot S_{\un,1,\F}
\end{align*}
in $K_0(\FF^\Delta_{\F pp_k})$. 
\end{lemma}
\begin{proof}
By Lemma~3.5 in \cite{BY24}, one has $\F\Proj(-,kG) = |\IBr(G)|\cdot S_{\un,1,\F}$. Since $\F\Proj(-,kGb)^T\subseteq \F\Proj(-,kGb)\subseteq \F\Proj(-,kG)$, it follows that $\F\Proj(-,kGb)^T$ and $\F\Proj(-,kG)$ are also multiples of $S_{\un,1,\F}$ in $K_0(\FF^\Delta_{\F pp_k})$.  Since the $\F$-dimension of $\F\Proj(\un,kGb)^T$ is equal to $|\PIM(G,b)^T|=|\IBr(G,b)^T|$, the result follows. 
\end{proof}

\begin{corollary}\label{cor GaloisAlperinFunctorial}
The following statements are equivalent:
\begin{enumerate}[{\rm (i)}]
	\item For any finite group $G$ and any block $b$ of $\bF_pG$, Conjecture \ref{conj: GAWC original version} holds.
\item For any finite group $G$, any block of $\bF_pG$ with a nontrivial defect group, and any subgroup $T$ of $\Gamma_b$, we have
$$\sum_{\tau\in[G\backslash\S_G]}(-1)^{|\tau|} \F \Proj(-,kG_\tau b_\tau)^T=0\quad \text{in } K_0(\FF^\Delta_{\F pp_k}),$$
where $G_\tau$ and $b_\tau$ are defined as in Notation \ref{notation:chain}.
\end{enumerate}
\end{corollary}
\begin{proof}
This follows easily from Lemma~\ref{lem Galoisstablefunctor} and Theorem~\ref{thm Galois refinement}.
\end{proof}

\begin{remark}
{\rm  For any finite group $G$ and a block idempotent $b$ of $\bF_pG$, Theorem~3.5 in \cite{BBY} implies that, in $K_0(\FF^\Delta_{\F pp_k})$,
	\begin{align*}
		\sum_{\tau\in[G\backslash\S_G]}(-1)^{|\tau|} \F T^\Delta(-,kG_\tau b_\tau)=\sum_{\tau\in[G\backslash\S_G]}(-1)^{|\tau|} \F \Proj(-,kG_\tau b_\tau)\,.
	\end{align*}
	Therefore, Corollary~\ref{cor GaloisAlperinFunctorial} can be seen as a Galois refinement of Theorem~3.5 in \cite{BBY}. 
}
\end{remark}

\section{Galois automorphisms and bijections between simple modules}\label{section: Funtorial equivalences and Brauer character bijections}

\begin{void}\label{void:Grothendieck group}
{\rm (i) Let $k$ be a field and let $R$ be a commutative ring with $1$. Let $A$, $B$ and $C$ be finite-dimensional $k$-algebras. For any $A$-module $M$, denote by $[M]$ the isomorphism class of $M$. Denote by $\S(A)$ the set of isomorphism classes of simple $A$-modules. Denote by $\R(A)$ the Grothendieck group of finite-dimensional $A$-modules with respect to short exact sequences. In other words, $\R(A)$ is a free $\Z$-module with a $\Z$-basis $\S(A)$. For any finite-dimensional $A$-module $M$, we have $[M]=[S_1]+\cdots+[S_n]$ in $\R(A)$, where $S_1,\cdots, S_n$ are the composition factors of $M$ (repeated according their multiplicities).  We set $\R(A,B):=\R(A\otimes_k B^{\rm op})$. Tensor products of bimodules induce a $\Z$-bilinear map $-\cdot_B-:\R(A,B)\times \R(B,C)\to \R(A,C)$, which extends to an $R$-bilinear map (denoted abusively by the same notation) $-\cdot_B-:R\otimes_\Z\R(A,B)\times R\otimes_\Z \R(B,C)\to R\otimes_\Z \R(A,C)$.
	
(ii) Let $k'$ be an extension of $k$ and let $A':=k'\otimes_k A$.  For an $A'$-module $U'$ and a $\sigma\in \Gal(k'/k)$, denote by $\lexp{\sigma}U'$ the $A'$-module which is equal to $U'$ as a module over the subalgebra $1\otimes A$ of $A'$, such that $x\otimes a$ acts on $U'$ as $\sigma^{-1}(x) \otimes a$ for all $a \in A$ and $x\in k'$.
We say that $U'$ {\it descends to} $k$, if there is an $A$-module $U$ such that $U'\cong k'\otimes_k U$.
In this special case, for any $\sigma\in\Gal(k'/k)$, the map sending $x\otimes u$ to $\sigma^{-1}(x)\otimes u$ induces an isomorphism $k'\otimes_k U\cong \lexp{\sigma}(k'\otimes_k U)$, where $u\in U$ and $x\in k'$. The action of $\Gal(k'/k)$ on $A'$-modules makes $\R(A')$ into a $\Z\Gal(k'/k)$-module, hence makes $R\otimes_\Z \R(A')$ into an $R\Gal(k'/k)$-module. It is clear that if $U'$ is simple, then $\lexp{\sigma}U'$ is simple. Hence $\S(A')$ is a $\Gal(k'/k)$-stable $\Z$-basis (resp. $R$-basis) of $\R(A')$ (resp. $R\otimes_\Z \R(A')$).
	
(iii)  Let $\eta$ be an element in $R\otimes_\Z \R(A,B)$. Then $\eta$ can be uniquely written as $\eta=\alpha_1[M_1]+\cdots+\alpha_s[M_s]$, where $M_i$ are finite-dimensional simple $A$-$B$-bimodules and $\alpha_i\in R$. In particular, $\eta$ induces an $R$-linear map 
	\begin{align*}
		\begin{split}
	\Phi_\eta:R\otimes_\Z \R(B)&\to R\otimes_\Z\R(A)\\
	[M] &\mapsto \eta\cdot_B[M]=\alpha_1[M_1\otimes_{kH}M] + \cdots+ \alpha_n[M_n\otimes_{kH} M]
	\end{split}
	\end{align*}
where $M$ is any finite-dimensional $kH$-module. Assume that there are  elements $\eta\in R\otimes_\Z \R(A,B)$ and $\omega\in R\otimes_\Z \R(B,A)$ such that 
\begin{align*}
	\eta \cdot_B\omega =[A]\,\, \text{in } R\otimes_\Z \R(A,A)\quad\text{and}\quad \omega \cdot_A\eta=[B]\,\, \text{in } R\otimes_\Z (B,B)\,.
\end{align*}
Then $\Phi_\eta$ and $\Phi_\omega$ are mutually inverse $R$-linear isomorphisms between $R\otimes_\Z \R(A)$ and $R\otimes_\Z \R(B)$. 

}
\end{void}

\begin{theorem}\label{theorem:Galois-equivariant bijections for general algebras}
Let $k\subseteq k'$ be an extension of fields and let $R$ be a commutative ring with $1$ containing $\Z$ as a unitary subring. Let $A$ and $B$ be finite-dimensional $k$-algebras. Let $A':=k'\otimes_k A$ and $B':=k'\otimes_k B$. Assume that $\Gal(k'/k)$ is cyclic and that there are  elements $\eta\in R\otimes_\Z \R(A,B)$ and $\omega\in R\otimes_\Z \R(B,A)$ such that 
\begin{align*}
	\eta \cdot_B\omega =[A]\,\, \text{in } R\otimes_\Z \R(A,A)\quad\text{and}\quad \omega \cdot_A\eta=[B]\,\, \text{in } R\otimes_\Z \R(B,B)\,.
\end{align*} Then there is a bijection $\S(A')\to \S(B')$ commuting with the action of $\Gal(k'/k)$.
\end{theorem}

\begin{proof}
The proof is inspired by the proof of \cite[Theorem 1]{H24}. Write $\eta=\alpha_1[M_1]+\cdots+\alpha_s[M_s]$ and $\omega=\beta_1[N_1]+\cdots+\beta_t[N_t]$, where $M_i$ (resp. $N_j$) are finite-dimensional simple $A$-$B$-bimodules (resp. $B$-$A$-bimodules) and $\alpha_i, \beta_j\in R$. Let 
$$\eta'=\alpha_1[k'\otimes_k M_1]+\cdots+\alpha_s[k'\otimes_k M_s]\in R\otimes_\Z \R(A',B')$$ 
and 
$$\omega'=\beta_1[k'\otimes_k N_1]+\cdots+\beta_t[k'\otimes_k N_t]\in R\otimes_\Z \R(A',B').$$
Then we have 
\begin{align*}
	\eta' \cdot_{B'}\omega' =[A']\,\, \text{in } R\otimes_\Z \R(A',A')\quad\text{and}\quad \omega' \cdot_{A'}\eta'=[B']\,\, \text{in } R\otimes_\Z (B',B')\,.
\end{align*}
Hence by \ref{void:Grothendieck group} (iii), $\Phi_{\eta'}$ and $\Phi_{\omega'}$ are mutually inverse $R$-linear isomorphisms between $R\otimes_\Z \R(A')$ and $R\otimes_\Z \R(B')$. 

Recall from \ref{void:Grothendieck group} (ii) that $R\otimes_\Z \R(A')$ and $R\otimes_\Z \R(B')$ are $R\Gal(k'/k)$-modules. Again by \ref{void:Grothendieck group} (ii), we have $\lexp{\sigma}\eta'=\eta'$ and $\lexp{\sigma}\omega'=\omega'$. Now we can easily check that 
$$\Phi_{\eta'}:R\otimes_\Z \R(B')\to R\otimes_\Z \R(A')$$
is an isomorphism of $R\Gal(k'/k)$-modules. Indeed, for any $\sigma\in \Gal(k'/k)$ and any $[M]\in \R(B')$ with $M$ being a finite-dimensional $B'$-module, one has
$$\lexp{\sigma}(\Phi_{\eta'}([M]))=\lexp{\sigma}(\eta'\cdot_{B'} [M])=\lexp{\sigma}\eta'\cdot_{B'} \lexp{\sigma}[M]=\eta'\cdot_B \lexp{\sigma}[M]=\Phi_{\eta'}(\lexp{\sigma}[M]).$$
Since $\S(B')$ is a $\Gal(k'/k)$-stable $R$-basis of $R\otimes_\Z \R(B')$, its image $\Phi_{\eta'}(\S(B'))$ is a $\Gal(k'/k)$-stable $R$-basis of $R\otimes_\Z R(A')$. Since the character $\chi$ of the $R\Gal(k'/k)$-module $R\otimes_\Z \R(A')$ is independent of the choice of an $R$-basis, we see that for any $\sigma\in \Gal(k'/k)$, we have 
$$\chi(\sigma)=|\S(A')^{\langle\sigma\rangle}|=|\Phi_{\eta'}(\S(B'))^{\langle\sigma\rangle}|.$$ 
Since $\Gal(k'/k)$ is a cyclic group, by Lemma \ref{lemma:permutation isomorphic}, $\S(A')$ and $\Phi_{\eta'}(\S(B'))$ are isomorphic as $\Gal(k'/k)$-sets. Hence $\S(A')$ and $\S(B')$ are isomorphic as $\Gal(k'/k)$-sets.
\end{proof}

\begin{proof}[Proof of Theorem \ref{theorem:funtorial equivalences and Brauer character bijections}]
Let $t=\exp(G\times H)_{p'}$, $\xi\in \bF_p$ a primitive $t$-th root of unity and $k'=\F_p[\xi]$. Since $k'$ is a splitting field for both $G$ and $H$, the values of any irreducible character of $G$ (resp. $H$) over $\bF_p$ are contained in $k'$, hence we have $\IBr(G,b)=\Irr(k'Gb)$ and $\IBr(H,c)=\Irr(k'Hc)$.  Let $\Gamma_t=\Gal(k'/\F_p)$. Then $(\Gamma_t)_b=\Gal(k'/k)=(\Gamma_t)_c$ is a cyclic group. Let $\sigma\in \Gamma_b=\Gal(\bF_p/k)$ and $\tilde{\sigma}$ be the image of $\sigma$ under the restriction map $\Gamma_b\to (\Gamma_t)_b$.  By Notation \ref{notation:characters of groups} (i) we have the following commutative diagram:
$$\xymatrix{
\S(\bF_pGb)\ar[r]\ar[d]^{\sigma} &	\IBr(G,b)\ar@{=}[r]\ar[d]^{\sigma} \ar[l]& \IBr(k'Gb)\ar[d]^{\tilde{\sigma}}\ar[r] & \S(k'Gb)\ar[l]\ar[d]^{\tilde{\sigma}} \\
\S(\bF_pGb)\ar[r] &	\IBr(G,b)\ar@{=}[r]\ar[l] &  \IBr(k'Gb) \ar[r]  &  \S(k'Gb)\ar[l]
}$$
Hence it suffices to show that $\S(k'Gb)$ and $\S(k'Hc)$ are isomorphic as $\Gal(k'/k)$-sets. By \ref{void:notaions related to functorial equiavlences} (iv), there are elements $\eta\in RT^\Delta(kGb,kHc)$ and $\omega\in RT^\Delta(kHc,kGb)$ such that 
\begin{align*}
	\eta \cdot_H\omega =[kGb]\,\, \text{in } RT^\Delta(kGb,kGb)\quad\text{and}\quad \omega \cdot_G\eta=[kHc]\,\, \text{in } RT^\Delta(kHc,kHc)\,.
\end{align*}
Hence we also have
\begin{align*}
	\eta \cdot_{kHc}\omega =[kGb]\,\, \text{in } R\otimes_\Z \R(kGb,kGb)\quad\text{and}\quad \omega \cdot_{kGb}\eta=[kHc]\,\, \text{in } R\otimes_\Z \R(kHc,kHc)\,.
\end{align*}
Now the statement follows from Theorem \ref{theorem:Galois-equivariant bijections for general algebras}.
\end{proof}

\begin{proof}[Proof of Corollary \ref{cor BrouetoGaloisAlperin}]
Since $P$ is abelian, it is well known that 
$$\bigsqcup_{Q\in[G\backslash\FF_G]} \W(N_G(Q)/Q,\bar{b}_Q)=\W((N_G(P)/P),\bar{b}_P),$$ and the inflation map $\W((N_G(P)/P),\bar{b}_P)\to \IBr(N_G(P),c)$ is a bijection. By definition, one easily sees that this bijection is commutative with the action of $\Gamma_b$. Now the statement follows from Theorem \ref{theorem:funtorial equivalences and Brauer character bijections}.
	\end{proof}

\section{Galois conjugate blocks are functorially equivalent}\label{section: Galois conjugate blocks are functorially equivalent}

Let $k$ denote an algebraically closed field of characteristic $p$ and $\F$ denote an algebraically closed field of characteristic zero. Recall that the category $\FF_{\F pp_k}^\Delta$ of diagonal $p$-permutation functors over $\F$ is semisimple, and the simple functors $S_{L,u,V}$ are parametrized by the isomorphism classes of triples $(L,u,V)$ where $(L,u)$ is a $\DD$-pair and $V$ is a simple $\F\Out(L,u)$-module.

Let $(G,b)$ and $(H,c)$ be group-block pairs over $k$. Since the category $\FF_{\F pp_k}^\Delta$ is semisimple, the functors $\F T^\Delta_{(G,b)}$ and $\F T^\Delta_{(H,c)}$ are direct sums of simple functors. It follows that the pairs $(G,b)$ and $(H,c)$ are functorially equivalent over $\F$ if and only if the multiplicities of the simple functor $S_{L,u,V}$ in $\F T^\Delta_{(G,b)}$ and $\F T^\Delta_{(H,c)}$ are the same. We now recall a formula for these multiplicities.

Let $(G,b)$ be a group-block pair over $k$ and let $(D,e_D)$ be a maximal $(G,b)$-Brauer pair. For any subgroup $P\le D$, let $e_P$ be the unique block idempotent of $kC_G(P)$ with $(P,e_P)\le (D,e_D)$. Let $\FF$ be the fusion system of $(G,b)$ with respect to $(D,e_D)$ and let $[\FF]$ denote a set of isomorphism classes of objects of $\FF$. 

Let $(L,u)$ be a $\DD$-pair. For $P\in\FF$, we denote by $\mathcal{P}_{(P,e_P)}(L,u)$ the set of group isomorphisms $\pi:L\to P$ with $\pi\circ u\circ \pi^{-1}\in\Aut_\FF(P)$. This is an $\left(N_G(P,e_P), \Aut(L,u)\right)$-biset via $$g\cdot \pi\cdot \varphi=i_g\pi\varphi$$ for $g\in N_G(P,e_P)$, $\pi\in\mathcal{P}_{(P,e_P)}(L,u)$ and $\varphi\in \Aut(L,u)$. We denote by $[\mathcal{P}_{(P,e_P)}(L,u)]$ a set of representatives of $N_G(P,e_P)\times \Aut(L,u)$-orbits of $\mathcal{P}_{(P,e_P)}(L,u)$. 

For $\pi \in [\mathcal{P}_{(P,e_P)}(L,u)]$, the stabilizer in $\Aut(L,u)$ of the $N_G(P,e_P)$-orbit of $\pi$ is denoted by $\Aut(L,u)_{\overline{(P,e_P,\pi)}}$. One has
\begin{align*}
	\Aut(L,u)_{\overline{(P,e_P,\pi)}}=\{\varphi\in \Aut(L,u)\mid \pi\varphi \pi^{-1}\in \Aut_{\FF}(P)\}\,.
\end{align*}

Moreover, for $\pi\in \mathcal{P}_{(P,e_P)}(L,u)$, we denote by $\PIM(kC_G(P)e_P,u)$ a set of isomorphism classes of projective indecomposable $kC_G(P)e_P$-modules that are fixed by $\pi u\pi^{-1}$. This is an $\Aut(L,u)_{\overline{(P,e_P,\pi)}}$-set via $U\cdot\varphi=\lexp{g}U$ where $g\in N_G(P,e_P)$ with $i_g\pi\varphi=\pi$. We denote by $\F\Proj(kC_G(P)e_P,u)$ the $\F$-span of $\PIM(kC_G(P)e_P,u)$. 

\begin{theorem}[{\cite[Theorem~8.22(b)]{BY22}}]\label{thm multiplicityformula}
	The multiplicity of a simple diagonal $p$-permutation functor $S_{L,u,V}$ in the functor $\F T^\Delta_{(G,b)}$ is equal to the $\F$-dimension of
	\begin{align*}
		\bigoplus_{P \in [\FF]} \bigoplus_{\pi \in [\mathcal{P}_{(P,e_P)}(L,u)]} \F\Proj(kC_G(P)e_P,u)\otimes_{\Aut(L,u)_{\overline{(P,e_P,\pi)}}} V\,.
	\end{align*}
\end{theorem}

Let $\F_p$ denote the prime subfield of $k$ and let $\Gamma=\Gal(k/\F_p)$ denote the Galois group. Note that $\Gamma$ acts on $kG$ and on $Z(kG)$ via $\F_p$-algebra automorphisms by applying $\sigma\in\Gamma$ to the coefficients of the elements of $kG$. We say that the group-block pairs $(G,b)$ and $(G,b')$ are \textit{$\Gamma$-conjugate}, if $b'=\sigma(b)$ for some $\sigma\in \Gamma$. 

Let $\sigma:k\to k$ be a field automorphism. It induces a ring isomorphism $\sigma:kG\to kG$ still denoted by $\sigma$. Note that the map $\sigma$ permutes the blocks of $kG$. Since the action of $\Gamma$ commutes with the action of $G$ and Brauer morphisms, it is easy to prove that Galois conjugate blocks have the same defect groups and the same fusion systems:

\begin{lemma}[cf. {\cite[Lemma~9.6.5]{Lin18b}}]\label{lem galoisconjugatepairs}
	Let $(G,b)$ be a group-block pair over $k$, let $(P,e)$ and $(Q,f)$ be Brauer pairs on $kG$ and let $\gamma\in \Gamma$. 
	\begin{enumerate}[(i)]
		\item $(P,e)$ is a $(G,b)$-Brauer pair if and only $(P,\sigma(e))$ is a $(G,\sigma(b))$-Brauer pair.
		\item We have $(Q,f)\le (P,e)$ if and only if $(Q,\sigma(f))\le (P,\sigma(e))$.
		\item $(P,e)$ is a maximal $(G,b)$-Brauer pair if and only if $(P,\sigma(e))$ is a maximal $(G,\sigma(b))$-Brauer pair. In this case, one has the equality $\mathcal{F}_{(P,e)}(G,b)=\mathcal{F}_{(P,\sigma(e))}(G,\sigma(b))$ of fusion systems. 
	\end{enumerate}
\end{lemma}

Now we are ready to prove Theorem \ref{theorem: Galois conjugate blocks are func equ}.

\begin{proof}[Proof of Theorem \ref{theorem: Galois conjugate blocks are func equ}]
	By the arguments above, it suffices to prove that the multiplicities of the simple diagonal $p$-permutation functors in $\F T^\Delta_{(G,b)}$ and $\F T^\Delta_{(G,\sigma(b))}$ are the same. We use the formula in Theorem~\ref{thm multiplicityformula}.
	
	Let $S_{L,u,V}$ be a simple diagonal $p$-permutation functor over $\F$. Let $(D,e_D)$ be a maximal $(G,b)$-Brauer pair and, for each $P\le D$, let $e_P$ be the unique block of $kC_G(P)$ with $(P,e_P)\le (D,e_D)$. Let $\FF$ denote the fusion system of $(G,b)$ with respect to $(D,e_D)$. By Lemma~\ref{lem galoisconjugatepairs}, $\FF$ is also the fusion system of the pair $(G,\sigma(b))$ with respect to the maximal $(G,\sigma(b))$-pair $(D,\sigma(e_D))$. 
	
 Since $\FF$ is a common fusion system, for every $P\in [\FF]$, one has $$\mathcal{P}_{(P,e_P)}(L,u)=\mathcal{P}_{(P,\sigma(e_P))}(L,u)\,.$$ Since also $N_G(P,e_P)=N_G(P,\sigma(e_P))$, we may choose representatives so that
	\begin{align*}
		\left[\mathcal{P}_{(P,e_P)}(L,u)\right]	=\left[\mathcal{P}_{(P,\sigma(e_P))}(L,u)\right]\,.
\end{align*}
Now fix $\pi\in \left[\mathcal{P}_{(P,e_P)}(L,u)\right]$. Then $\pi i_u\pi^{-1}=i_g$ for some $g\in N_G(P,e_P)=N_G(P,\sigma(e_P))$. A projective indecomposable $kC_G(P)e_P$-module $S$ is $u$-fixed if and only if 
\begin{align*}
	\lexp{g}S\cong S \quad\text{as } kC_G(P)e_P\text{-modules}.
\end{align*}
Applying $\sigma$ to this isomorphism gives
\begin{align*}
	\lexp{\sigma}{\left(\lexp{g}S\right)}\cong \lexp{\sigma}S \quad\text{as } kC_G(P)\sigma(e_P)\text{-modules}.
\end{align*}
Since $\lexp{\sigma}{\left(\lexp{g}S\right)}=\lexp{g}{\left(\lexp{\sigma}S\right)}$, it follows that the map $S\mapsto \lexp{\sigma}S$ is a bijection between the sets $\PIM(kC_G(P)e_P,u)$ and $\PIM(kC_G(P)\sigma(e_P),u)$. This bijection induces an isomorphism of $\F$-vector spaces
\begin{align*}
	\F\Proj(kC_G(P)e_P,u) \cong \F\Proj(kC_G(P)\sigma(e_P),u)\,.
\end{align*}
Next, we show that this is an isomorphism of right $\F\Aut(L,u)_{\overline{(P,e_P,\pi)}}$-modules. Since $\Aut(L,u)_{\overline{(P,e_P,\pi)}}=\Aut(L,u)_{\overline{(P,\sigma(e_P),\pi)}}$, Theorem~\ref{thm multiplicityformula} will then imply the result. 

Let $\varphi\in \Aut(L,u)_{\overline{(P,e_P,\pi)}}$. Let $g\in N_G(P,e_P)=N_G(P,\sigma(e_P))$ such that $ i_g\pi\varphi =\pi$. For any $U\in \F\Proj(kC_G(P)e_P,u)$, one has
\begin{align*}
	\lexp{\sigma}{(U\cdot\varphi)}=\lexp{\sigma}{(\lexp{g}U)}=\lexp{g}{(\lexp{\sigma}U)} =\lexp{\sigma}U\cdot \varphi\,.
\end{align*}
Thus the isomorphism
\begin{align*}
	\F\Proj(kC_G(P)e_P,u) \cong \F\Proj(kC_G(P)\sigma(e_P),u)
\end{align*}
is an isomorphism of right $\F\Aut(L,u)_{\overline{(P,e_P,\pi)}}$-modules, as desired.
\end{proof}


\bigskip\noindent\textbf{Acknowledgements.}\quad The first author is supported by National Natural Science Foundation of China (12501024, 12471016), Fundamental Research Funds for the Central Universities (CCNU24XJ028), and China Postdoctoral Science Foundation (GZC20252006, 2025T001 HB). The second author gratefully acknowledges the hospitality of Central China Normal University during a visit in August 2025, when this work was initiated.

\end{document}